\documentclass{amsart}

\newtheorem{theorem}{Theorem}[section]
\newtheorem{lemma}[theorem]{Lemma}

\theoremstyle{definition}

\theoremstyle{remark}
\newtheorem{remark}{Remark}

\numberwithin{equation}{section}

\usepackage{setspace}

\allowdisplaybreaks

\begin{document}

\title{Remarks on the extension of the Ricci flow}

\author{Fei He}
\address{Department of Mathematics, University of California, Irvine, CA 92697}

\email{hef@uci.edu}

\date{\today}

%\dedicatory{}

%\keywords{}

\begin{abstract}
We present two new conditions to extend the Ricci flow on a compact manifold over a finite time, which are improvements of some known extension theorems.
\end{abstract}

\thanks{This research was partially supported by NSF grant DMS-0801988}

\maketitle

\section{Introduction}
We say that $g(t)$ is a Ricci flow solution if it satisfies the following equation defined by Richard Hamilton \cite{HR1} :
\begin{equation}\label{ricci flow}
 \frac{\partial}{\partial t}g(t)= -2 Ric(t) .
\end{equation}
In the following $Rm(t)$, $Ric(t)$ and $R(t)$ denote the Riemann, Ricci and scalar curvature tensors of $g(t)$ respectively, and $|Rm(t)|$, $|Ric(t)|$ denote the corresponding norms.

The Ricci flow equation (\ref{ricci flow}) has been studied extensively. Short time existence of solutions was first established by R. Hamilton in \cite{HR1}. Convergence of solutions to Einstein metrics is proved possible for initial metrics with special curvature conditions. In general, the Ricci flow solution will develop singularities in finite or infinite time. Therefore an important topic in the theory is the behavior of curvature tensors at a singular time.

%In \cite{HR1}, R. Hamilton proved the short-time existence of a smooth solution to equation (\ref{ricci flow}) on closed Riemannian manifolds with arbitrary initial metric, and applied it to evolve a Riemannian metric of positive Ricci curvature on a closed 3-dimensional manifold to a space form.
%After Hamilton's work, The Ricci flow equation has been studied extensively. Convergence of solutions to Einstein metrics is proved possible for initial metrics with special curvature conditions. In general, the solution of (\ref{ricci flow}) will develop singularities in finite or infinite time (\cite{HR2}). Therefore it is important to study curvature behaviors when the flow approaches a singularity.

R. Hamilton showed in \cite{HR2} that if $T<\infty$ is a finite singular time, we have
\begin{equation*}
 \limsup_{t \to T} \sup_M |Rm(t)|= +\infty.
\end{equation*}
\noindent
In other words, if the sectional curvature is uniformly bounded on a finite time interval, then the flow can be extended (\cite{HR1}). The proof is by establishing Bernstein-Bando type smoothing estimates using the maximum principle. Hamilton's theorem has been improved by Natasa Sesum who showed that if
\begin{equation*}
\sup_{M \times [0,T)} | Ric(t) |<\infty,
\end{equation*}
then the Ricci flow can be extended past time $T < \infty$ (\cite{SN}). These results are known as extension theorems for the Ricci flow.

A natural question is: what is the weakest curvature condition to extend the Ricci flow? There is a conjecture that in a finite time singularity of the Ricci flow, the supremum of the scalar curvature will blow up. This conjecture is confirmed for Type I singularities (\cite{EMT}, \cite{SeLe3}) and for the K\"ahler Ricci flow (\cite{Zhang}). But the general case is still open. One of the best known results in this direction is the following Theorem \ref{Bing Wang theorem}, which is implied by the proof method of Natasa Sesum in \cite{SN}, and is also proved by Bing Wang in \cite{WB}:
\begin{theorem}\label{Bing Wang theorem}
Suppose $(M, g(t)),0\leq t < T < \infty,$ is a Ricci flow solution on a closed manifold. If
\begin{equation*}
\int_0^T \sup_M |Ric(t)| dt < +\infty .
\end{equation*}
Then the flow can be extended past time $T$.
\end{theorem}

\begin{remark}
 Bing Wang also proved in \cite{WB} that there is a gap for $\limsup_{t\to T}{|T-t|\sup_M|Ric(t)|}$ where $T$ is the singular time.
\end{remark}

Theorem \ref{Bing Wang theorem} tells that $\sup_M |Ric(t)| $ not only blows up at a finite singular time, but also has to grow fast enough so that its integral on the maximal existence time interval is infinite. This clearly recovers previous results mentioned above.

Our first theorem is a further improvement in this direction. The new idea is to explore the optimal Sobolev constant and apply a related rigidity theorem.

\begin{theorem}\label{main result 1}
Suppose$(M, g(t)),0\leq t < T,$ is a Ricci flow solution on a closed manifold, $T < \infty$. If the function $F(x):= \int_0^T|Ric(x,t)|dt$ is continuous on M, then the flow can be extended past time $T$.
\end{theorem}

\begin{remark}
Note that if $\int_0^T \sup_M |Ric(t)| dt  <\infty $, the dominated convergence theorem implies the continuity of $\int_0^T |Ric(x,t)| dt$, and we recover Theorem \ref{Bing Wang theorem}.
\end{remark}

The proof of Theorem \ref{main result 1} uses a blow-up argument. Recall that by Hamilton's compactness theorem (\cite{HR3}) and Perelman's no-local-collapsing theorem (\cite{Perelman}), we can choose a sequence of times and points $(x_i, t_i) \in M \times [0,T), i=1,2,...$, where $t_i \to T$, such that the sequence of dilated pointed solutions $(M,g_i(t), x_i)$ with $g_i(t)$ defined by
\begin{equation*}
g_i(t):= |Rm(x_i, t_i)|g(t_i+ \frac{t}{|Rm(x_i,t_i)|})
\end{equation*}
converges in the pointed Cheeger-Gromov sense to a complete limit solution $(M_\infty, $ $g_\infty (t), x_\infty)$, $t\in (-\infty, \omega)$, where $\omega $ is a positive number or $\infty$. It's important that this limit solution is non-flat when $T$ is a finite singular time, in particular $|Rm(x_\infty, 0)|=1$. This compactness result is very useful in studying the behavior of the Ricci flow at a singular time. For example, recall that (\cite{SN}) Natasa Sesum studied the volume growth of geodesic balls in $(M_\infty,g_\infty(0))$, and used the rigidity part of the volume comparison theorem to conclude that, if $|Ric(t)|$ is uniformly bounded for $t \in [0,T)$, then $(M_\infty,g_\infty(0))$ is isometric to the Euclidean space, hence contradicting with the non-flatness.

Under the assumption of Theorem \ref{main result 1}, we can establish an optimal Euclidean Sobolev inequality on $(M_\infty, g_\infty(0))$, then apply the following rigidity theorem of M. Ledoux to show that $(M_\infty,g_\infty(0))$ is isometric to the Euclidean space.

\begin{theorem}[M. Ledoux,\cite{Ledoux}]\label{Ledoux}
Let $(M,g)$ be a smooth, complete $n$-dimensional Riemannian manifold with nonnegative Ricci curvature. Suppose that for some $q\in [0,n)$, the Sobolev inequality
\begin{equation*}
\left(\int_M |u|^p d\mu \right)^{q/p} \leq K(n,q)^q \int_M |\nabla u|^q d\mu
\end{equation*}
is valid for $\forall u\in C_0^\infty (M)$, where $1/p=1/q-1/n$, $K(n,q)$ is the optimal Sobolev constant for the Euclidean space. Then $(M,g)$ is isometric to $(\mathbb{R}^n, g_{\text{\tiny{flat}}})$.
\end{theorem}

\begin{remark}
 The value $K(n,q)$ is computed in \cite{T}.
\end{remark}
 Optimal constants in Sobolev inequalities have been studied by many authors, and one can refer to \cite{Hebey} for a comprehensive exposition. In the proof we need a theorem of T. Aubin \cite{Aubin}:

\begin{theorem}[T. Aubin]\label{Aubin}
Let (M,g) be a smooth, compact Riemannian $n-$manifold. For any $\epsilon> 0$ and any $q\in [1,n)$ real, there exists $B\in \mathbb{R}$ such that for any $u \in H_1^q(M)$,
\begin{equation*}
\left(\int_M |u|^p d\mu \right)^{q/p} \leq (K(n,q)^q+\epsilon) \int_M |\nabla u|^q d\mu + B\int_M u^q d\mu
\end{equation*}
where $1/p=1/q-1/n$ and $K(n,q)$ is the optimal Sobolev constant for the Euclidean space.
\end{theorem}

\begin{remark}
T. Aubin's theorem has been improved by E. Hebey and M.Vaugon (\cite{HV1}, \cite{HV2}, ) who showed that the $\epsilon$ in Theorem \ref{Aubin} can be removed, in both compact and complete settings.
\end{remark}

\begin{remark}
T. Aubin's theorem implies that for any $\epsilon >0$, we have a family of Sobolev inequalities in the form
\begin{equation*}
\left( \int_M |u|^{2n/(n-2)} d\mu(t) \right)^{(n-2)/n} \leq (K(n,2)^2+\epsilon) \int_M |\nabla u|^2_{g(t)} d\mu(t) + B(t)\int_M u^2 d\mu(t),
\end{equation*}
along the flow. The proof of Theorem \ref{main result 1} implies that $B(t)$ must blow up at a finite singular time. It will be very interesting to get an upper bound estimate of $B(t)$ explicitly in terms of curvature, however such an estimate is not yet available to our knowledge.
\end{remark}

Space-time integral bounds on the curvature have also been considered by many authors. In \cite{WB2} B. Wang proved that if
\begin{equation}\label{critical}
  \int_0^T\int_{M} |Rm|^{p} d\mu dt < \infty,\quad p\geq \frac{n}{2}+1,
\end{equation}
then the Ricci flow can be extended past time $T$. Similar results are also proved in \cite{Ye} by R. Ye and in \cite{MaCheng} by L. Ma and L. Cheng. Note that the power $\frac{n}{2}+1$ is critical, which makes the integral scaling invariant. If $p < \frac{n}{2}+1$, the integral in (\ref{critical}) can be bounded even when $T$ is a singular time, as demonstrated by the shrinking sphere (See Example 2.1 in \cite{WB2}).

For the mean curvature flow, the same extension problem has also been studied. The supremum and certain scaling invariant space-time integrals of the norm of the second fundamental form are known to blow up at a finite singular time (\cite{Huisken}, \cite{SeLe1}, \cite{SeLe2}, \cite{XYZ}). Moreover, the surprising fact that a subcritical integral quantity has to blow up was proved by N. Le in \cite{Le}:

\begin{theorem}[N.Le]\label{Le}
 Let $A(t)$ be the second fundamental form of a $n-$dimensional compact hyper-surface without boundary $M_t$ in $\mathbb{R}^{n+1}$ evolving by the mean curvature flow. If
 \begin{equation*}
  \int_0^T\int_{M_t} \frac{|A|^{n+2}}{\log(1+|A|)} d\mu dt < \infty,
 \end{equation*}
 then the flow can be extended past time $T$.
\end{theorem}
Note that $n+2$ is the critical power in the mean curvature flow case, and that the integral quantity in N. Le's theorem is sub-scaling invariant.

One of the key elements in the proof is the Michael-Simon inequality, which one uses to establish a Sobolev inequality, then applying Nash-Moser iteration to prove the following mean-value type inequality:
\begin{equation}\label{mean value ineq}
 \sup_{M_t} |A(t)|\leq C(M_0) \left( 1+ \int_0^T\int_{M_t} |A|^{n+3} \right).
\end{equation}
Theorem \ref{Le} then follows an elementary calculus method.

Our second result is a Ricci flow version of N. Le's theorem.

\begin{theorem}\label{subcritical}
 Let $(M,g(t)), t\in[0,T)$ be a Ricci flow solution. If
 \begin{equation*}
  \int_0^T\int_{M} \frac{|Rm|^{n/2+1}}{\log(1+|Rm|)} d\mu dt < \infty,
 \end{equation*}
 Then the flow can be extended past time $T$.
\end{theorem}

In the Ricci flow case, we can use a blow-up argument and apply the `doubling-time estimate' (Corollary 7.5 in \cite{ChowKnopf}) to establish an inequality similar to (\ref{mean value ineq}), then use the same calculus method to prove Theorem \ref{subcritical}.

%------------------------------------------------------------------------------------------------------------------------
\section{Proof of Theorem \ref{main result 1}}
\begin{proof}[Proof of Theorem \ref{main result 1}]
We claim that under the assumption of the theorem, the sectional curvature $|Rm|$ is bounded, hence the flow can be extended by Hamilton's result (Theorem 14.1 in \cite{HR1}).

If the claim is not true, we can choose a sequence of times and points $(x_i, t_i) \in M \times [0,T), i=1, 2, ...$, such that  $Q_i:=|Rm(x_i,t_i)|\to \infty $ and $t_i \to T$ as $i \to \infty$, and the sequence of dilated pointed solutions $(M,g_i(t), x_i)$ with $g_i(t)$ defined by
\begin{equation*}
g_i(t):= Q_ig(t_i+ \frac{t}{Q_i})
\end{equation*}
converges in the pointed Cheeger-Gromov sense to a non-flat limit solution $(M_\infty, g_\infty,\\
 x_\infty)$ (See Chapter 8 of \cite{ChowKnopf}). In the following we use $\phi_i, i=1,2,...$ to denote the diffeomorphisms in the pointed Cheeger-Gromov limit ( See Chapter $3$ of \cite{CHOW1} for a detailed definition). Also, we use $R^+$ and $R^-$ to denote the positive and negative parts of the scalar curvature, and $\lambda$ is the negative part of the smallest eigenvalue of the Ricci curvature.

\indent
We first prove that ($M_\infty, g_\infty$) has nonnegative Ricci curvature (We actually prove it is Ricci-flat). By the continuity assumption on $F(x):= \int_0^T|Ric(x,t)|dt$ and the compactness of $M$, we can use elementary arguments to prove that
\begin{equation*}
 \lim_{s \to T} \int_s^T |Ric(x,t)| dt =0 \quad \text{uniformly for}\quad \forall x \in M.
\end{equation*}
Then we compute
\begin{eqnarray*}
\int_{-1}^0 |Ric_{g_\infty(t)}|(x) dt &=& \lim_{i \to \infty} \int_{-1}^0 |Ric_{\phi_i^* g_i(t)}|(x) dt       \\
                                       &=& \lim_{i \to \infty} \int_{-1}^0 |Ric_{ g_i(t)}|(\phi_i (x)) dt      \\
                                       &=& \lim_{i \to \infty} \int_{t_i-1/{Q_i}}^{t_i} |Ric_{ g(t)}|(\phi_i (x)) dt \\
                                       &\leq&  \lim_{i \to \infty} \int_{t_i-1/{Q_i}}^{t_i} |Ric_{g(t)}|(\phi_i(x))dt \\
                                       &\leq&  \lim_{i \to \infty} \int_{t_i-1/{Q_i}}^{T} |Ric_{g(t)}|(\phi_i(x))dt \\
                                       &=& 0.
\end{eqnarray*}
Which implies $|Ric_{g_\infty(t)}|(x)=0, \forall x \in M_\infty, \forall t \in [-1,0]$.

Next we establish a Sobolev inequality on $M_\infty$. Observe that
\begin{equation*}
\frac{d}{dt} |\nabla u|^2_g(t)(x)= 2 Ric(t)(\nabla u, \nabla u),
\end{equation*}
and
\begin{equation*}\label{derivative of volume form}
\frac{d}{dt}d\mu_{g(t)}(x)= - R(x,t)d\mu_{g(t)}(x) .
\end{equation*}
These imply that
\begin{equation*}
|\nabla u|^2(x,t_0)\leq |\nabla u|^2(x,t_1) e^{2\int_{t_0}^{t_1}\lambda(x,t)dt},
\end{equation*}
and
\begin{equation*}
e^{-\int_{t_0}^{t_1}R^+(x,t)dt} d\mu(x,t_0) \leq d\mu(x,t_1) \leq e^{\int_{t_0}^{t_1}R^-(x,t)dt} d\mu(x,t_0).
\end{equation*}

Now we need the continuity of $F(x):= \int_0^T|Ric(x,t)|dt$ and the compactness of $M$ again. For any $\epsilon > 0$, by elementary analysis we can find $t_0(\epsilon)$ such that $\forall t_2> t_1 \geq t_0$, we have
\begin{eqnarray*}
0\leq \int_{t_1}^{t_2} R^-(x,t)dt \leq n\int_{t_1}^{t_2} |Ric(x,t)| dt<\epsilon;\\
0\leq \int_{t_1}^{t_2} R^+(x,t)dt \leq n\int_{t_1}^{t_2} |Ric(x,t)|dt <\epsilon; \\
0\leq \int_{t_1}^{t_2} \lambda(x,t)dt \leq \int_{t_1}^{t_2} |Ric(x,t)|dt <\epsilon;
\end{eqnarray*}
for all $x\in M$.

Theorem \ref{Aubin} implies that we have a Sobolev inequality at the time $t_0$:
\begin{equation*}
\left( \int_M |u|^{2n/(n-2)} d\mu(t_0) \right)^{(n-2)/n}
\leq (K(n,2)^2+\epsilon) \int_M |\nabla u|^2_{g(t_0)} d\mu(t_0)
                                             + B(t_0)\int_M u^2 d\mu(t_0),
\end{equation*}
for any $ u \in H^2_1(M)$.
Then the above observation implies that for any $t_1\in [t_0, T) $, $(M, g(t_1))$ has a Sobolev inequality:
\begin{eqnarray*}
\left( \int_M |u|^{2n/(n-2)} d\mu(t_1) \right)^{(n-2)/n} &\leq& (K(n,2)^2+\epsilon) e^{(3-2/n)\epsilon} \int_M |\nabla u|^2_{g(t_1)} d\mu(t_1) \\
                                            && + B(t_0)e^{(2-2/n)\epsilon}\int_M u^2 d\mu(t_1),
\end{eqnarray*}
for any $u \in H^2_1(M)$.

Now we pick any $u \in C_0^\infty(M_\infty)$, and suppose $u$ is supported on a compact domain $V$. The idea is to push $u$ forward by $\phi_i$ to a function on $(M, g(t_i))$ for each $i$ s.t. $t_i> t_0$, apply the Sobolev inequality, then pull back to $(M_\infty, g_\infty(0))$ and take the limit in $i$. We compute:
\begin{eqnarray*}
&& \left( \int_{M_\infty} |u|^{2n/(n-2)} d\mu_{g_\infty(0)} \right)^{(n-2)/n} \\
&=& \lim_{i\to \infty} \left( \int_V |u|^{2n/(n-2)} d\mu_{\phi_i^* g_i(0)} \right)^{(n-2)/n} \\
                                                              &=& \lim_{i\to \infty} \left( Q_i^{n/2}\int_{\phi_i(V)} |u\circ \phi_i^{-1}|^{2n/(n-2)} d\mu_{ g(t_i)} \right)^{(n-2)/n} \\
                                                              &\leq& \lim_{i\to \infty} Q_i^{(n-2)/n} \Big[ (K(n,2)^2+\epsilon) e^{(3-2/n)\epsilon} \int_{\phi_i(V)} |\nabla (u\circ \phi_i^{-1})|^2_{g(t_i)} d\mu_{g(t_i)}  \\
                                                              &&+ B(t_0)e^{(2-2/n)\epsilon}\int_{\phi_i(V)} u^2 d\mu_{g(t_i)}   \Big]                      \\
                                                              &=& \lim_{i\to \infty} \Big[ (K(n,2)^2+\epsilon) e^{(3-2/n)\epsilon} \int_{\phi_i(V)} |\nabla (u\circ \phi_i^{-1})|^2_{g_i(0)} d\mu_{g_i(0)} \\
                                                              &&+\frac{B(t_0)e^{(2-2/n)\epsilon}}{Q_i}\int_{\phi_i(V)} u^2 d\mu_{g_i(0)} \Big]                         \\
                                                              &=& (K(n,2)^2+\epsilon) e^{(3-2/n)\epsilon} \int_{M_\infty} |\nabla u|^2_{g_\infty(0)} d\mu_{g_\infty(0)} .    \\
\end{eqnarray*}
Since $\epsilon$ is arbitrary, we can let it go to zero. Then we establish the optimal Euclidean Sobolev inequality
\begin{equation*}
\left( \int_{M_\infty} |u|^{2n/(n-2)} d\mu_{g_\infty(0)} \right)^{(n-2)/n} \leq K(n,2)^2 \int_{M_\infty} |\nabla u|^2_{g_\infty(0)} d\mu_{g_\infty(0)}
\end{equation*}
on $(M_\infty, g_\infty(0))$.

By Theorem \ref{Ledoux}, $(M_\infty, g_\infty(0))$ is isometric to the Euclidean space, contradicting the non-flatness of $g_\infty(0)$.

\end{proof}

%--------------------------------------------------------------------------------------------------------------

\section{Proof of Theorem \ref{subcritical}}
To prove Theorem \ref{subcritical}, we first establish a similar inequality to  (\ref{mean value ineq}) by a compactness argument. We need the following `doubling-time estimate', which is Corollary 7.5 in \cite{ChowKnopf}:

\begin{lemma}[Doubling-time estimate]\label{doubling-time estimate}
There exists $c(n)$ depending only on $n$, such that if $(M,g(t),t\in[0,T))$ is a Ricci flow solution on a compact manifold of dimension $n$, then
\begin{equation*}
\sup_M |Rm(t)| \leq 2\sup_M |Rm(0)| \quad \text{for all times} \quad t\in [0,\frac{c(n)}{\sup_M|Rm(0)|}).
\end{equation*}
\end{lemma}

\begin{lemma}\label{normalizedmvi}
 Let $\mathcal{M}=\{ g(t)| t\in[0,1], \quad g(t) \quad \text{has non-collapsing constant } \linebreak
 \kappa , \quad  \sup_M |Rm(0)|\leq C_0\}$ be a set of nonsingular Ricci flow solutions on a closed $n$-dimensional manifold $M$. There exists a constant $C(n,\kappa, C_0)$ such that for any $g(t) \in \mathcal{M}$
 \begin{equation*}
  \sup_{M\times[0,1]}|Rm| \leq C \int_0^1\int_M |Rm|^{n/2+2}d\mu dt + 2C_0.
 \end{equation*}
\end{lemma}

\begin{proof}
 If not, we can find a sequence $g_i(t), i=1,2,...$ in $\mathcal{M}$, such that
 \begin{equation*}
  \sup_{M\times[0,1]}|Rm_i| \geq P_i \int_0^1\int_M |Rm_i|^{n/2+2}d\mu_i dt + 2C_0,
 \end{equation*}
 where $P_i \to +\infty$ as $i \to \infty$. Let $Q_i=\sup_{M\times[0,1]}|Rm_i|$ for each $i$, then we can find $(x_i,t_i)$ such that $Q_i=|Rm_i(x_i,t_i)|$. Note that $Q_i > 2C_0$, Lemma \ref{doubling-time estimate} implies that $t_i \geq c(n)/C_0$, hence $Q_it_i\geq 2c(n)$. Dilate this sequence
 \begin{equation*}
  \tilde{g}_i(t)=Q_i g_i(t_i+t/Q_i), -t_iQ_i\leq t\leq(1-t_i)Q_i, i=1,2,...
 \end{equation*}
 The dilated solutions $(M, \tilde{g_i}(t), x_i)$ has a common existence interval $[-2C_0, 0]$, a uniform bound on the curvature and an injectivity radius lower bound by the assumption on the non-collapsing constant $\kappa$. Hence they converge in the pointed Cheeger-Gromov sense to a limit solution $(M_\infty, \tilde{g}(t), x_\infty), t\in [-2C_0,0]$, with $|\widetilde{Rm}|(x_\infty,0)=1$.
 But we can compute on any $x_\infty \in \Omega\subset M_\infty$
 \begin{eqnarray*}
  && \int_{-2C_0}^0\int_\Omega |\widetilde{Rm}|^{n/2+2} d\tilde{\mu} dt \\
  &=& \lim_{i\to \infty}1/Q_i \int_{t_i-2c(n)/Q_i}^{t_i} \int_{\phi_i(\Omega)}|Rm_i|^{n/2+2}d\mu_i dt \\
  &\leq& \lim_{i \to \infty} \left( \frac{1}{Q_i} \frac{Q_i-2C_0}{P_i} \right) \\
  &=& 0.
 \end{eqnarray*}
 Which implies that $|\widetilde{Rm}|(x_\infty,0)=0$, contradiction!
\end{proof}

\begin{lemma}[Mean Value Inequality]\label{prop}
 For a Ricci flow solution $(M,g(t)), t\in[0,T), T< \infty$, there exists constants $C_0(n,\kappa, \sup_M|Rm(0)|)$ and\\ $C_1=T\max\{2\sup_M|Rm(0)|,2\sup_M|Rm(0)|^2/c(n)\}$, where $c(n)$ is the constant in the `doubling-time estimate', such that for any $t\in[0,T)$
 \begin{equation*}
  \sup_{M\times[0,t]}|Rm| \leq C_0 \int_0^t\int_M |Rm(x,s)|^{n/2+2}d\mu ds + C_1.
 \end{equation*}
\end{lemma}

\begin{proof}
 We only need to prove the lemma for non-trivial solutions. Without loss of generality let $T=1$.

 For $t\in[0,c(n)/\sup_M|Rm(0)|)$ it's clearly true by Lemma \ref{doubling-time estimate}.

 For any $t\in[c(n)/\sup_M|Rm(0)|,1)$, define
 \begin{equation*}
 \tilde{g}(s)=\frac{1}{t} g(ts),\quad s\in[0,1].
 \end{equation*}
  Then
  \begin{equation*}
  |\widetilde{Rm}(0)|\leq t|Rm(0)|\leq |Rm(0)|.
  \end{equation*}
  Note that the non-collapsing constant $\kappa$ is scaling invariant. Lemma \ref{normalizedmvi} implies
 \begin{equation*}
  \sup_{M\times[0,t]}|\widetilde{Rm}| \leq C_0 \int_0^t\int_M |\widetilde{Rm}(x,s)|^{n/2+2}d\tilde{\mu} ds + 2\sup_M|Rm(0)|.
 \end{equation*}
 Then we scale it back to the original metric $g(t)$. Since the scaling factor $t$ is now bounded below by $c(n)/\sup_M|Rm(0)|$, we get
 \begin{equation*}
   \sup_{M\times[0,t]}|Rm| \leq C_0 \int_0^t\int_M |Rm(x,s)|^{n/2+2}d\mu ds + 2\sup_M|Rm(0)|^2/c(n).
 \end{equation*}
\end{proof}

Now we can use the same method as in \cite{Le} to prove Theorem \ref{subcritical}.
\begin{proof}[Proof of Theorem \ref{subcritical}]
 Let
  \begin{equation*}
   f(t)=\sup_M|Rm(t)|,
  \end{equation*}
  \begin{equation*}
   G(t)=\int_M \frac{|Rm|^{n/2+1}}{\log(1+|Rm|)}d\mu(t),
  \end{equation*}
  and
  \begin{equation*}
   \psi(s)=s\log(1+s).
  \end{equation*}
  Then $\psi$ is an increasing function when $s\geq 0$. By Lemma \ref{prop}, for any $t\in [0,T)$
 \begin{eqnarray*}
  f(t) &\leq& C\int_0^t\int_M \psi(|Rm|)\frac{|Rm|^{n/2+1}}{\log(1+|Rm|)} d\mu d s + C_1 \\
       &\leq& C\int_0^t\psi(f(s)) G(s) d s +C_1 \\
       &=:& h(t).
 \end{eqnarray*}
 $h^\prime(t)= C \psi(f(t))G(t) \leq C \psi(h(t))G(t)$ since $\psi$ is nondecreasing. Then we have
 \begin{eqnarray*}
  \int_{h(0)}^{h(t)}\frac{1}{\psi(s)} d s &=& \int_0^t C G(t) dt \\
                                          &\leq& C\int_0^t\int_{M} \frac{|Rm|^{n/2+1}}{\log(1+|Rm|)} d\mu dt \\
                                          &<& \infty.
 \end{eqnarray*}

 Since $\int_1^\infty \frac{1}{\psi(s)} ds = \infty$, we deduce that $\sup_{[0,T)}h(t)< \infty$, hence $\sup_{[0,T)}f(t)< \infty$. Therefore the flow can be extended by Theorem 14.1 in \cite{HR1}.
\end{proof}

%-------------------------------------------------------------------------------------------------------------------------
\textbf{Acknowledgements:} The author would like to thank his advisor Peter Li for his advising, encouragement and generous support. Also thank Jeffrey Streets for many useful suggestions and all his help in preparing this paper, and thank Guoyi Xu for helpful discussions.

%--------------------------------------------------------------------------------------------------------------------------

\bibliographystyle{amsplain}

\end{document}